\documentclass[11pt,a4paper]{amsart}
\usepackage[french, ngerman, italian, english]{babel}
\usepackage{amsmath,amsthm, amssymb, amsfonts}
\usepackage{setspace}
\usepackage{anysize}
\usepackage{url}
\usepackage{graphicx}
\usepackage[colorlinks=false]{hyperref}
\usepackage{paralist}
\usepackage{verbatim}
\usepackage[utf8]{inputenc}

\theoremstyle{plain}
\newtheorem{theorem}{\bf Theorem}[section]

\newtheorem*{theorem*}{Theorem}
\newtheorem{proposition}[theorem]{\bf Proposition}
\newtheorem{lemma}[theorem]{\bf Lemma}
\newtheorem{corollary}[theorem]{\bf Corollary}

\theoremstyle{definition}
\newtheorem{definition}[theorem]{\bf Def\mbox{}inition}
\newtheorem{remark}[theorem]{\bf Remark}
\theoremstyle{remark}
\newtheorem{example}[theorem]{\bf Example}
\theoremstyle{example}

\title{Veronese Algebras and Modules of Rings with Straightening Laws}
\author{Alexandru Constantinescu}

\address{Dipartimento di Matematica, Universit\`a di Genova,
Via Dodecaneso 35, 16146 Genova, Italy}
\email{constant@dima.unige.it}
\date{}
\def\k#1{\ensuremath{K[x_{1} , \ldots , x_{#1}]}}
\def\HF{Hilbert function}
\def\GB{Gr$\ddot{\textrm{o}}$bner basis}
\def\ini{\operatorname{\rm in}}

\def\R#1{\ensuremath{R^{(#1)}}}
\def\M#1#2{\ensuremath{M_{#1}^{(#2)}}}
\def\A#1{\ensuremath{A^{(#1)}}}
\def\P#1{\ensuremath{P^{(#1)}}}

\def\rank{\textup{rank}}
\def\supp{\textup{supp}}
\def\hght{\textup{ht}}
\def\tor{\mathop{\rm Tor}\nolimits}
\def\a{\ensuremath{\alpha}}
\def\b{\ensuremath{\beta}}
\def\c{\ensuremath{\gamma}}
\def\d{\ensuremath{\delta}}
\def\m{\ensuremath{\mu}}
\def\vv{\texttt{v}}
 
\begin{document}
\maketitle
\begin{abstract}
 Do the Veronese rings  of an algebra with straightening laws (ASL) still have an  ASL structure? 
We give positive answers to this question in  
some particular cases, namely for the second Veronese algebra of Hibi rings and  of discrete ASLs. We also prove that the Veronese modules of the polynomial ring have a structure of module with straightening laws. In dimension at most three we 
 present a poset construction  that has the required combinatorial properties to support such a structure.
 \end{abstract}
\section*{Introduction}

The notion of algebra with straightening laws (ASL for short) was introduced in the early eighties by  De Concini, Eisenbud and  Procesi in \cite{DEP}. These algebras give an unif\mbox{}ied treatment of both algebraic and geometric objects that have a combinatorial nature. The coordinate rings of some classical algebraic varieties, such as determinantal rings (in particular the coordinate ring of Grassmannians) and Pfaf\mbox{}f\mbox{}ian rings, are  examples of ASL's. In \cite{bruns},  Bruns generalizes this notion in a natural way, by  introducing the concept of module with straightening laws (MSL) over an ASL. For a general overview on this subject the reader may consult \cite{DEP}, the books of Bruns and Vetter \cite{bv} and of Bruns and Herzog \cite{bh}.

One interesting question regarding ASLs is whether their Veronese algebras still have a structure of algebra with straightening laws. So far, the only positive answer to this question was given in \cite{Conca1}  by Conca in the case of the polynomial ring. The ASL structure described in \cite{Conca1} indicates that this question cannot have a simple answer.  The main idea behind the ASL structure is to  give a partial order on a set of K-algebra generators in such a way that the relations among these generators are \lq\lq compatible\rq\rq~with the partial  order. As we will see in this paper, one of the main obstacles to overcome in the search for an answer to the above question is of combinatorial nature. In particular, one needs to construct a new poset that should support the ASL structure of the Veronese algebra. 

In the f\mbox{}irst section of this paper we will introduce   the terminology and notation that we will use later on. We will also present a few known results that will turn out useful. 
%
%

In Section 2 we will extend Conca's result, namely we will prove that the Veronese modules of the polynomial ring have a structure of MSL. Here, the ASL structure of the polynomial ring given in \cite{Conca1} plays an important role. As a corollary we will obtain the result of  Aramova,  B\u arc\u anescu and Herzog (see \cite{abh}) which states  that the Veronese modules have a linear resolution.  Using the results of Bruns from \cite{bruns2} on MSL's, we will be able to  give an upper bound for the rate of a f\mbox{}initely generated MSL. 

In the third section we will study  the Veronese algebra of a homogeneous ASL. The f\mbox{}irst step towards proving that it is again an ASL is to construct a new poset. Using the translation of algebraic properties into combinatorial ones, we can  sketch the prof\mbox{}ile of the poset that needs  to be constructed. Unfortunately, we were not able to f\mbox{}ind a construction that works in general. However, we will prove that the second Veronese algebra of a discrete ASL and of a Hibi ring is again an ASL. The poset that we construct is the second zig-zag poset.

In the last section of this paper we will construct  a new  poset starting from a poset of rank three.  Then we will prove that it has the combinatorial properties to support an ASL structure of the Veronese algebra. 

The author wishes to thank his advisor Aldo Conca for his encouragement and for his helpful remarks on preliminary versions of this paper. We also  thank  J\"urgen Herzog for suggesting the study of the Veronese modules.

\section{Preliminaries} 

 Let us summarize the basic def\mbox{}initions and terminology that we will use. 
 Throughout this paper we will consider only  f\mbox{}inite partially ordered sets (posets). Let $P$ be a poset and let  $C : \a_1 < \ldots < \a_t$ be a chain in $P$ (i.e. a totally ordered subset of $P$). With this notation we say that $C$ is a chain \emph{descending} from $\a_t$.  The \emph{length} of $C$ will be the cardinality of the set $C$.
 The \emph{rank} of a poset $P$, denoted by $rank(P)$, is the supremum of the lengths of all chains contained in $P$. A poset is called \emph{pure} if all maximal chains have the same length. 
 The \emph{height} of an element $\a \in P$, denoted $\hght(\a)$ is:  
 \[\hght(\a) = \sup\{ \textup{length of chains descending from } \a\} - 1.\]
 Given a natural number $m \ge 1$, an \emph{$m$-multichain} in $P$ is a weakly  increasing sequence of $m$ elements of $P$: $\a_1 \le \ldots \le \a_m$.
 A \emph{poset ideal} of $P$ is a subset $I$ such that if $\a \in I, \b \in P$ and $\b \le \a$ then $\b \in I$.
 
 Let $K$ be f\mbox{}ield, $A$ be a ring and $P \subset A$ be a poset.
 We call a  \emph{monomial} a product of the form $\a_1\a_2 \ldots \a_t$ where $\a_i \in P,~~ \forall~ i$. A monomial  $\a_1\a_2 \ldots \a_t$ is called \emph{standard} if $\a_1 \le \a_2 \le \ldots \le \a_t$.
 We will use the def\mbox{}inition of an ASL which is also used by Bruns in  \cite{bruns}. This def\mbox{}inition is given for graded $K$-algebras, but one can def\mbox{}ine an ASL also in the non-graded case (see \cite{DEP,hibi}).

 \begin{definition}
 Let  $A$ be a $K$-algebra and $P \subset A$ a f\mbox{}inite poset. We say that $A$ is  a (\emph{graded})\emph{ algebra with straightening laws} on $P$ over $K$ if the following conditions are satisf\mbox{}ied:
\begin{itemize}
\item[(ASL 0)] $A = \bigoplus_{i\ge0} A_i$ is a graded $K$-algebra such that $A_0 = K$, the elements of $P$ are homogeneous  of positive degree and they generate $A$ as a $K$-algebra.
\item[(ASL 1)] The set of standard monomials is a basis of $A$ as  a $K$-vector space.
\item[(ASL 2)] (Straightening Laws) If $\a, \b \in P$ are incomparable  elements (written $\a \not\sim \b$) and if 
\begin{equation}\label{SL}
\a\b = \sum_i r_i \c_{i1} \c_{i2} \ldots \c_{it_i},
\end{equation}
 with  $0 \ne r_i \in K$ and $\c_{i1} \le \c_{i2} \le \ldots\le \c_{it_i}$, is the unique linear combination of standard monomials given by (ASL 1), then $\c_{i1} < \a$ and $\c_{i1} < \b$ for every $i$.
 \end{itemize}
 When $P \subset A_1$ we say that $A$ is a \emph{homogeneous} ASL over $P$.
 \end{definition}

Note that in (\ref{SL}) the right hand side can be equal to $0$, but that, even though $1$ is a standard monomial, no $\c_{i1} \c_{i2} \ldots \c_{it_i}$ can be $1$. These relations are called the \emph{straightening laws} (or straightening relations) of $A$.

An ASL $A$ on $P$,  can be presented as $K[P]/I$, where $K[P]$ is the polynomial ring whose variables are the elements of $P$ and $I$ is the homogeneous ideal generated by the straightening laws. Denote by $I_P$ the monomial ideal of $K[P]$ generated by $\a\b$ with $\a, \b\in P$ and $\a\not\sim\b$. A linear extension of $(P,<)$ is a total order $<_1$ on $P$ such that $\a < \b$ implies $\a <_1 \b$ for any $\a, \b \in P$. When $A$ is a homogeneous ASL on $P$ and $\tau$ is the reverse lexicographic  term ordering with respect to a linear extension of $<$, the polynomials given in (ASL 2) form a \GB~ of $I$ and  the initial ideal of $I$ with respect to $\tau$ is $\ini_{\tau}(I) = I_P$. 
The algebra $K[P]/I_P$ is an ASL on $P$ and it is called the \emph{discrete} ASL.

The discrete ASL over a poset $P$ can be also seen  as the Stanley-Reisner ring of the simplicial complex $\Delta_P$, where $\Delta_P$ is the  complex whose vertices are the elements of $P$ and whose facets are the maximal chains of $P$. This is a useful remark, as it allows one to compute the \HF~ of any ASL on $P$ by looking at the $f$-vector of $\Delta_P$.

The following proposition is  easy to check, but nevertheless very useful:

\begin{proposition}\label{subposet}
Let $A$ be an ASL on $P$ over $K$, and $H \subset P$ a poset ideal of $P$. Then the ideal $AH$ is generated as a $K$-vector space by the standard monomials containing a factor $\a \in H$, and $A/AH$ is an ASL on $P \setminus H$ (where $P \setminus H$ is embedded in $A/AH$ in a natural way). 
\end{proposition}
\noindent This proposition allows one to prove results on ASLs using induction on the cardinality of $P$. Also the ASL structure in many examples is established this way.\\

The notion of ASL has a natural generalization  to modules in the following sense. For a module $M$ over an ASL $A$ we want the generators of $M$ to be partially ordered, a distinguished set of "standard elements" should form a $K$-basis of $M$ and the multiplication $A \times M \longrightarrow M$ should satisfy a straightening law similar to the straightening law of $A$. We have the following def\mbox{}inition due to Bruns:

\begin{definition} Let $A$ be an ASL on $P$ over a field $K$. An $A$-module $M$ is called \emph{module with straightening laws} on a f\mbox{}inite poset $Q \subset M$ if the following conditions are satisf\mbox{}ied:
\begin{itemize}
\item[(MSL 1)] For every $x \in Q$ there exists a poset ideal $\mathcal{I}(x) \subset P$ such that the elements
\begin{displaymath}
\a_1\a_2\ldots\a_i x, \quad with~  \a_1 \notin \mathcal{I}(x),\quad \a_1\le\a_2\le\ldots\le\a_i ~\textup{and}~ i\ge 0,
\end{displaymath}
form a basis of $M$ as a $K$-vector space. These elements are called \emph{standard elements}.
\item[(MSL 2)] For every $x \in Q$ and $\a \in \mathcal{I}(x)$  one has 
\begin{equation}
\a x \in \sum_{y < x} Ay.
\end{equation}
\end{itemize}
\end{definition}
\noindent An MSL on a poset $Q$ over a homogeneous ASL, say $A$, is called \emph{homogeneous} if it is a graded $A$-module in which $Q$ consists of elements of degree 0.
From (MSL 1) and (MSL 2) it follows immediately by induction on the rank of $x$ that each element $\a x$ with $\a \in \mathcal{I}(x)$ has a standard representation 
\[ \a x = \sum_{y<x}\Big(\sum r_{\a x \m y} \m \Big) y, \qquad \textup{with}~ 0 \ne r_{\a x \m y} \in K,  \]
in which every $\m y$ is a standard element.

\begin{remark} 
a) If $M$ is a MSL on a poset $Q$ and  $Q' \subset Q$ is a poset ideal, then the submodule of $M$ generated by $Q'$ is a MSL too. This allows one to prove theorems on MSLs by noetherian induction \mbox{ on the set of ideals of $Q$.}

b) In the def\mbox{}inition of MSL it would have been enough to require that the standard elements are linearly independent, because (MSL 2) and the induction principle above guarantee that $M$ is generated as a $K$-vector space by the standard elements.
\end{remark}
Given a graded $K$-algebra $A = \bigoplus_{i\ge0} A_i$ and $d \ge 2$ a positive integer, the \emph{$d$-Veronese algebra} of $A$ is by def\mbox{}inition $$A^{(d)} = \bigoplus_{i\ge0} A_{di}.$$ 
For every $d \ge 2$ one can consider for every $0 \le j \le d-1$ the \emph{$j$-thVeronese module} of $A$: $M_j^{(d)} = \bigoplus_{i\ge0} A_{di + j}.$ The module $M_j^{(d)}$ is obviously an $A^{(d)}$-module. \\

The polynomial ring in $n$ variables $R = \k{n}$ has an ASL structure by taking $x_1,\ldots,x_n$ as generators and the order: $x_1 \le \ldots \le x_n$.
 In \cite{Conca1},  the author  proves that the Veronese algebra of the polynomial ring  is still an ASL when the f\mbox{}ield $K$ is inf\mbox{}inite. 
The monomials in $n$ variables of degree $d$ are a natural choice for the generators of $R^{(d)}$. Unfortunately, already when $n=2$ and $d=3$, one cannot partially order the set $\{x_1^3, x_1^2x_2, x_1x_2^2, x_2^3\}$ in a compatible way with an ASL structure for $K[x_1,x_2]^{(3)}$. 

In order to f\mbox{}ind an  ASL  structure for $R^{(d)}$ one has to proceed as follows. 
For $i = 1,\ldots,n$ and $j = 1,\ldots, d$, take $\ell_{i,j}$ to be generic linear forms such that for any $j_1, \ldots, j_n \in \{1,\ldots,d\}$ the linear forms $\ell_{1,j_1} , \ldots ,\ell_{n,j_n}$ are linearly independent. The assumption on the cardinality of the field $K$ is needed for the existence of such forms. 
Take as generators of $R^{(d)}$ all products $\ell_{s_11}\ldots\ell_{s_dd}$ with the property  that $\sum_{i=1}^{i=d}s_i \le n-d+1$. 
Order these generators as follows: $$\ell_{s_11}\ldots\ell_{s_dd} \le \ell_{t_11}\ldots\ell_{t_dd} \iff s_i \le t_i \textup{ for every } i.$$
The abstract poset $H_n(d)$ corresponding to the partial order def\mbox{}ined on this set of generators is obtained as follows. Denote by $H(d) = \{1,\dots,n\}^d$ and order its elements component-wise, i.e.  $(\a_1,\ldots,\a_d) \le (\b_1,\ldots,\b_d)  \iff  \a_i \le \b_i, ~\forall~i$. 
Denote by $H_n(d)$   the subposet of $H(d)$ of elements of rank $\le n$, that is 
 \[ H_n(d) = \{ (\a_1,\ldots,\a_d)\in H(d) ~:~ \sum_{i=1}^d \a_i \le n+d-1\}.\]
 For our goals we will not need a description of the straightening relations on these generators.   We will only use the fact that $R^{(d)}$ has an  ASL structure  on $H_n(d)$.
For more details see Conca's paper \cite{Conca1}. 
Here is a graphical representation (Hasse diagram) of the poset $H_n(d)$ for $d=2$ and $d =3$, when $n =3$:
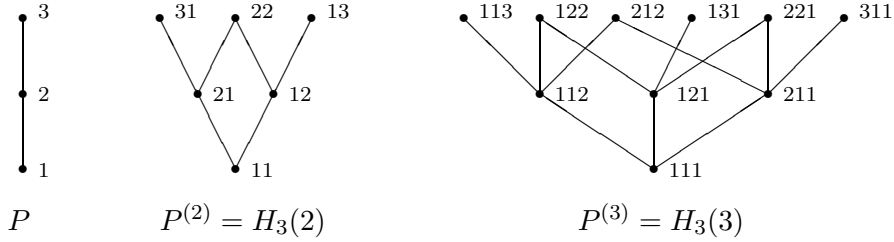
\[\setlength{\unitlength}{1mm}
\begin{picture}(112,40)
{\scriptsize
\put(2,10){\circle*{1}}   \put(4,9){$1$}
\put(2,20){\circle*{1}}   \put(4,19){$2$} 
\put(2,30){\circle*{1}}   \put(4,30){$3$}
}
\put(0,2){$P$}

\put(2,10){\line(0,1){20}}
\put(30,10){\line(-1,2){10}}
\put(30,10){\line(1,2){10}}
\put(25,20){\line(1,2){5}}
\put(35,20){\line(-1,2){5}}
{\scriptsize
\put(30,10){\circle*{1}}  \put(32,9){$11$}
\put(25,20){\circle*{1}}  \put(37,19){$12$}
\put(35,20){\circle*{1}}  \put(27,19){$21$}
\put(20,30){\circle*{1}}  \put(22,30){$31$}
\put(30,30){\circle*{1}}  \put(32,30){$22$}
\put(40,30){\circle*{1}}  \put(42,30){$13$}
}
\put(20,2){$P^{(2)} = H_3(2)$}
{\scriptsize
\put(85,10){\circle*{1}}  \put(87,9){$111$}
\put(70,20){\circle*{1}}  \put(72,19){$112$}
\put(85,20){\circle*{1}}  \put(88,19){$121$}
\put(100,20){\circle*{1}}  \put(102,19){$211$}
\put(60,30){\circle*{1}}  \put(62,30){$113$}
\put(70,30){\circle*{1}}  \put(72,30){$122$}
\put(80,30){\circle*{1}}  \put(82,30){$212$}
\put(90,30){\circle*{1}}  \put(92,30){$131$}
\put(100,30){\circle*{1}}  \put(102,30){$221$}
\put(110,30){\circle*{1}}  \put(112,30){$311$}
}
\put(75,2){$P^{(3)} = H_3(3)$}

\put(70,20){\line(-1,1){10}}
\put(70,20){\line(0,1){10}}
\put(70,20){\line(1,1){10}}

\put(85,10){\line(-3,2){15}}
\put(85,10){\line(0,1){10}}
\put(85,10){\line(3,2){15}}

\put(85,20){\line(-3,2){15}}
\put(85,20){\line(1,2){5}}
\put(85,20){\line(3,2){15}}

\put(100,20){\line(-2,1){20}}
\put(100,20){\line(0,1){10}}
\put(100,20){\line(1,1){10}}
\end{picture}
\]

A useful remark is that the 2nd Veronese algebra of the polynomial ring $R$ has an ASL structure also with the usual  monomials as generators. In this case the field $K$ may have any  cardinality. Consider the following order on the set of variables: $x_1 < x_2 < \ldots < x_n$. We can  order the degree two monomials as follows:
\[
x_ix_j \le x_kx_l \iff x_i \le x_k ~~\textup{and}~~ x_j \ge x_l.
\]
The straightening laws will be given in the following way.
If $x_ix_j \not\sim x_kx_l$, then rearrange the indices $i, j, k, l$ in increasing order, say $i_1 \le j_1 \le k_1 \le l_1$ (i.e. $\{i,j,k,l\} = \{i_1,j_1,k_1,l_1\}$ as multisets) and def\mbox{}ine:
\[
(x_ix_j)(x_kx_l) = (x_{i_1}x_{l_1})(x_{j_1}x_{k_1}).
 \]
 It is clear that $(x_{i_1}x_{l_1})(x_{j_1}x_{k_1})$ is a standard monomial. Also one can easily check that these relations are exactly the relations of the second Veronese algebra. In this case the new poset is the second zig-zag poset $Z_2(P)$ (see Section 3 for def\mbox{}inition), but it is also isomorphic to $H_n(2)$.
 For example, if $n = 3$, the poset will look like this:\\
 \setlength{\unitlength}{1mm}
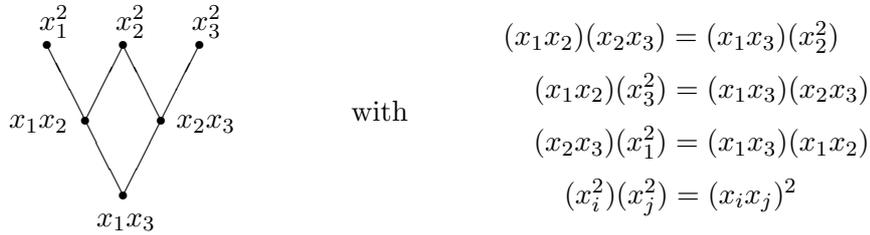
\begin{picture}(80,40)(-30,0)
\put(10,10){\line(-1,2){10}}
\put(10,10){\line(1,2){10}}
\put(5,20){\line(1,2){5}}
\put(15,20){\line(-1,2){5}}
\put(10,10){\circle*{1}}  \put(6.5,6){$x_1x_3$}
\put(5,20){\circle*{1}}  \put(-5,19){$x_1x_2$}
\put(15,20){\circle*{1}}  \put(17,19){$x_2x_3$}
\put(0,30){\circle*{1}}  \put(-1,32){$x_1^2$}
\put(10,30){\circle*{1}}  \put(9,32){$x_2^2$}
\put(20,30){\circle*{1}}  \put(19,32){$x_3^2$}
\put(40,20){with}
\put(60,30){$(x_1x_2)(x_2x_3) = (x_1x_3)(x_2^2)$}
\put(64,23){$(x_1x_2)(x_3^2) = (x_1x_3)(x_2x_3)$}
\put(64,16){$(x_2x_3)(x_1^2) = (x_1x_3)(x_1x_2)$}
\put(68,9){$(x_i^2)(x_j^2)  = (x_ix_j)^2$}
\end{picture}\\
where $i \neq j$ and $i,j \in \{1,2,3\}$.

\section{The MSL Structure of the  Veronese Modules}

In this section we will prove that the Veronese modules of the polynomial ring have a structure of MSLs as $R^{(d)}$-modules.  For this part only we will assume that the field $K$ is infinite. We will then see that the MSL structure implies that the Veronese modules have a linear resolution.  Finally  we will find an upper bound for the rate of a  f\mbox{}initely generated  MSL. The bound is given  in terms of the degrees of its generators and the degrees of the generators of the ASL. 

Let $d\ge2$ be a positive integer, $j \in \{0 \ldots d-1\}$ and assume that the filed $K$ is infinite. Consider the same generic linear forms that give the ASL structure of $R^{(d)}$, presented in the previous section. Choose  as generators of $M_j^{(d)}$ the products of the form:
 \[\ell_{i_11}\cdots\ell_{i_jj}, \quad \textrm{with~} i_1 + \ldots + i_j \le n+j-1.\] 
Order them component-wise, just as in the case of the Veronese algebra of $R$. So the poset supporting the MSL structure will be $H_n(j)$.
To simplify notation, we will denote the generators of $R^{(d)}$, respectively the generators of $M^{(d)}_j$, by: 
\[\begin{array}{rclll}
f_{\a_1 \ldots \a_d} &= &\ell_{\a_11}\ldots \ell_{\a_dd}~, &\quad \forall~ (\a_1,\ldots,\a_d) &\textrm{~with~} \sum_{i=1}^d\a_i \le n+d-1,\\  
\rule{0 pt}{3ex}
g_{i_1 \ldots i_j} &= &\ell_{i_11}\ldots \ell_{i_jj}~, &\quad \forall~ (i_1,\ldots,i_j) &\textrm{~with~} \sum_{k=1}^j i_k \le n+d-1.
\end{array}\] 
As  \R{d} is generated as a $K$-algebra by  the $f_{\a_1 \ldots \a_d}$-s for every $d$, we get that the $g_{i_1 \ldots i_j}$-s generate $M_j^{(d)}$ as a \R{d}-module.
To every such generator we  associate a poset ideal of $H_n(d)$ as follows:
\[
\mathcal{I}(g_{i_1 \ldots i_j}) = 
\{f_{\a_1 \ldots \a_d}~:~ (\a_1,\ldots,\a_j,\ldots,\a_d) \not\ge (i_1,\ldots, i_j,1,\ldots,1)\}. 
\]
 It is clear that $\mathcal{I}(g_{i_1 \ldots i_j})$ is  a poset ideal for any $g_{i_1 \ldots i_j}$. We will prove the following:

\begin{theorem}
Let $R = \k{n}$ be the polynomial ring in $n$ variables. For every $d\ge 2$ and for every $j \in \{0,\ldots,d-1\}$, the $j$-th Veronese module $M_j^{(d)}$ is a homogenous MSL on $H_n(j)$ over $R^{(d)}$ with the structure def\mbox{}ined above.
\end{theorem} 
 
 \begin{proof}
 If $j=0$ then $M_0^{(d)} = R^{(d)}$ as $R^{(d)}$-modules. In this case, we have a trivial MSL structure over the poset $Q = \{1\}$, with $\mathcal{I}(1) = \phi$ (see \cite{bruns}). So we will suppose from now  on that $j \ge 1$.
 
 In order to prove that we have an MSL structure for \M{j}{d} we have to check the following:
\begin{itemize}
\item[1.]
$ \forall~ g_{i_1 \ldots i_j}$ and $\forall~ f_{\a_1 \ldots \a_d} \in \mathcal{I}(g_{i_1 \ldots i_j})$ we have:
 \[f_{\a_1 \ldots \a_d}\cdot g_{i_1 \ldots i_j} \in \sum_{g_{k_1 \ldots k_j} < g_{i_1 \ldots i_j}} R^{(d)}\cdot g_{k_1 \ldots k_j}.\]
\item[2.] The standard elements are linearly independent over $K$.
\end{itemize}

To prove 1. let us choose  $g_{i_1 \ldots i_j}$ for some $(i_1,\ldots,i_j) \in H_n(j)$ and some $f_{\a_1 \ldots \a_d} \in \mathcal{I}(g_{i_1 \ldots i_j})$. This means that $(\a_1,\ldots,\a_d) \not\ge (i_1,\ldots, i_j,1,\ldots,1)$. So there exists an index $s \in \{1,\ldots, j\}$ such that $\a_s < i_s.$ We have:
\[
\begin{array}{ccl}
f_{\a_1 \ldots \a_d} \cdot g_{i_1 \ldots i_j} &= &\ell_{\a_11}\ldots \ell_{\a_dd}\cdot( \ell_{i_11}\ldots \ell_{i_j j})\\
 &\rule{0pt}{3ex}{= }& \ell_{\a_11} \ldots \ell_{i_s s} \ldots \ell_{\a_d d} \cdot (\ell_{i_11} \ldots \ell_{\a_s s} \ldots \ell_{i_j j})\\
 &\rule{0pt}{3ex}{= }& \ell_{\a_11} \ldots \ell_{i_s s} \ldots \ell_{\a_d d} \cdot g_{i_1 \ldots,\a_s, \dots i_j}
\end{array}
\]
As  $\a_s < i_s$ we also have that $g_{i_1 \ldots\a_s \dots i_j} < g_{i_1 \ldots i_s \ldots i_j}$, so 1. holds true.

As all standard elements are homogeneous polynomials, in order to prove the second part, we have to look only at linear combinations of standard elements of the same degree. 
Let $F$ be a linear combination of standard elements of degree $md + j$:
\[ F = \sum \lambda \m g_{i_1 \ldots i_j},
\]
where not all $\lambda \in K$ are zero and every  $\mu = f_{\a_{11} \ldots \a_{1d}}\cdot\ldots\cdot  f_{\a_{m1} \ldots \a_{md}}$ is a standard monomial in \R{d} with  $f_{\a_{11} \ldots \a_{1d}} \notin \mathcal{I}(g_{i_1 \ldots i_j})$. In particular 
$(\a_{11}, \ldots, \a_{1 j},\ldots, \a_{1d}) \ge (i_1, \ldots ,i_j,1,\ldots,1)$ for all $g_{i_1 \ldots i_j}$. 
If $F = 0$ then also $F \cdot \ell_{1 j+1}\ldots\ell_{1d} = 0$. But for all $g_{i_1 \ldots i_j}$ we have 
\[
g_{i_1 \ldots i_j} \cdot \ell_{1 j+1}\ldots\ell_{1d} = f_{i_1 \ldots i_j 1 \ldots 1}. 
\]
As $f_{i_1 \ldots i_j 1 \ldots 1} \le f_{\a_{11} \ldots \a_{1d}} \le \ldots \le  f_{\a_{m1} \ldots \a_{md}}$ we have that 
\[
F \cdot \ell_{1 j+1}\ldots\ell_{1d} = \sum \lambda f_{i_1 \ldots i_j 1 \ldots 1}  f_{\a_{11} \ldots \a_{1d}} 
\ldots   f_{\a_{m1} \ldots \a_{md}} = 0
\]
is a  linear combination of standard monomials in \R{d}. So, as the standard monomials form a $K$-basis of $R^{(d)}$, all the coef\mbox{}f\mbox{}icients $\lambda$ must be zero.

As  \R{d} is a homogenous ASL, $M_j^{(d)}$ is a graded \R{d}-module and we chose generators of degree zero for  $M_j^{(d)}$, by def\mbox{}inition we obtain a homogenous MSL.
\end{proof}

\noindent As a consequence of the homogeneous MSL structure, by \cite[Theorem 1.1]{bhv}, we obtain the following result of Aramova,  B\u arc\u anescu and Herzog   \cite[Theorem 2.1]{abh}:
\begin{corollary}
The $R^{(d)}$-module $M_j^{(d)}$ has a linear resolution for every $j \in \{0,\ldots,d-1\}$.\\
\end{corollary}

In the last part of this section we will prove a result regarding the Betti numbers of a module with straightening laws. We will then see that this result has a nice consequence for the rate of the module. From this point on the field $K$ may have any cardinality. 
For any ASL $A$ (not necessarily homogeneous) on a poset $P$  and  any MSL $M$ on a poset $Q$ over $A$, we know by \cite[(2.6)]{bruns2}  there exists a f\mbox{}iltration of $M$:
\[0 = M_0 \subset M_1 \subset \ldots \subset M_r =M, \]
with $M_{l+1}/M_l \cong A/A\mathcal{I}(q)$, for some $q \in Q$. The modules $M_l$ are actually the $A$-modules generated by $q_1,\ldots, q_l$, where $q_1\le q_2 \le \ldots \le q_r$ are all the elements of $Q$ ordered by a linear extension of the partial order on $Q$. Using this f\mbox{}iltration and the fact that $A\mathcal{I}(q)$ is a MSL over $A$ (see \cite[Example 3.1]{bruns}), we are able to prove the following.

\begin{proposition}\label{grez}
Let $A$ be an ASL on $P$ over a f\mbox{}ield $K$ and $M$ be a MSL on $Q$ over $A$. Denote by $d = \max\{ \deg(p)~:~ p \in P\}$ and by $m = \max\{ \deg(q)~:~ q\in Q\}$. We have:
\[  \b_{i,j}(M) = 0, \quad \textrm{for all~~} i,j \textrm{~~with~~} j-i \ge i(d-1) + m +1, \]
where $\b_{i,j}(M)= \dim_K\tor_i^A(M,K)_j  $ denote the graded Betti numbers of $M$.
\end{proposition}
\begin{proof}
We will use induction on $i$ and on the cardinality $|Q|$ of the poset. If $i = 0$ everything is clear. We will see in the proof that  for each $i$ the case  $|Q| = 1$  follows  only from inductive hypothesis  on $i-1$.

Let $i > 0$ and let $Q=\{q_1,\ldots,q_l\}$ with $0< l$ be a poset with its elements written in an order given  by a linear extension of the partial order. Suppose that for $i-1$  the assumption holds  for any poset $Q'$ and that for $i$ the assumption holds if $|Q'|<l$. For simplicity  we will denote throughout this proof  $\mathcal{I}_{q_l} := \mathcal{I}(q_l)$ and $m_l:=\deg(q_l)$. In order to make  the following exact sequence homogenous, we have to twist $A/A\mathcal{I}_{q_l}$ by $\deg(q_l)$: 
\[ 0 \longrightarrow M_{l-1} \longrightarrow M_l \longrightarrow M_l/M_{l-1} \cong A/A\mathcal{I}_{q_l}(-m_l) \longrightarrow 0.\] 
So we obtain the exact sequence
\begin{equation}\label{tor}
 \tor_i^A(M_{l-1},K)_j \longrightarrow\tor_i^A(M_{l},K)_j\longrightarrow \tor_i^A(A/A\mathcal{I}_{q_l}(-m_l),K)_j.
 \end{equation}
From the short exact sequence
\[ 0 \longrightarrow A\mathcal{I}_{q_l} \longrightarrow A \longrightarrow A/A\mathcal{I}_{q_l}(-m_l) \longrightarrow 0\] 
we obtain that 
\[ \tor_i^A(A/A\mathcal{I}_{q_l}(-m_l),K)_j =  \tor_{i-1}^A(A\mathcal{I}_{q_l}(-m_l),K)_j\] 
(this is why the case $ |Q| = 1$ follows only from induction on $i$). From \cite[Example 3.1]{bruns} we know that  $A\mathcal{I}_{q_l}$ is an ASL on the subposet $\mathcal{I}_{q_l} \subset P$. So by induction on $i$ we get that
\[\tor_i^A(A/A\mathcal{I}_{q_l}(-m_l),K)_j  = 0,\]
  if   $j-m_l - (i-1) \ge (i-1)(d-1) + d +1.$
It is clear that this  is equivalent to $j-i \ge i(d-1) + m_l +1$, so as $m_l\le m$ we obtain:
\[\tor_i^A(A/A\mathcal{I}_{q_l}(-m_l),K)_j  = 0, \quad\textrm{ if~}  j-i\ge i(d-1) + m +1.\]
To the left of $\tor_i^A(M_{l},K)_j$ in (\ref{tor}), by induction on the cardinality of the poset, we have that:
\[ \tor_i^A(M_{l-1},K)_j = 0,\quad \textrm{if~} j - i \ge i(d-1) + m +1,\]
and this completes the proof.
\end{proof}
In \cite{back},  Backelin introduced for any homogenous $K$-algebra $A$ a numerical invariant called the \emph{rate of $A$}. This invariant measures how much $A$ deviates from being Koszul. In \cite{abh}, the authors def\mbox{}ine the \emph{rate}  for any f\mbox{}initely generated $A$-module in the following way. As $\tor_i^A(M,K)$ is a f\mbox{}initely generated $K$-vector space, one may set
\[t_i(M) = \sup\{j~:~  \tor_i^A(M,K)_j \neq 0\}; \]
and then  def\mbox{}ine the rate of $M$ as
\[ \textup{rate}_A(M) = \sup_{i \ge 1} \Big\{ \frac{t_i(M)}{i}\Big\}.\]
Note that $t_i(M)$ is the highest shift in the $i$-th position of the minimal free homogenous resolution of $M$. 
With this def\mbox{}inition, Proposition \ref{grez} has the following corollary:

\begin{corollary}
If $M$ is a MSL over the ASL $A$, with the above notations we have:
\[ \textup{rate}_A(M) \le d+m.\]
\end{corollary}

\section{The Veronese Algebra of an ASL}

In this section we will study whether the Veronese algebra of a homogeneous ASL still has  a structure of algebra with straightening laws.
We have seen that so far the only known case is that of the polynomial ring. The complicated structure of its Veronese algebra as an ASL indicates that this question does not have an easy answer. 

Let us f\mbox{}irst see what we should be looking for. Given $A$ a homogeneous ASL on $P$ over $K$, we want to f\mbox{}ind  poset \P{d} such that \A{d} has an ASL structure on \P{d} over $K$. Translating  the algebraic properties of \A{d}  into combinatorial properties, we can outline the characteristics  that a possible candidate for \P{d} should have. Here are some known facts about ASLs:
\begin{itemize}
\item[(1a)] If $A$ is an ASL on a poset $P$ over $K$ and $A$ is integral then $P$ has an unique minimal element.
\item[(2a)] The Krull dimension of $A$ is equal to the rank of $P$.
\item[(3a)] The \HF~ of a homogeneous ASL $A$ on $P$ can be computed directly from the poset $P$ in the following way:
\[\dim_K A_i = |\{ \textrm{multichains of length $i$ in P}\}|.\]
\end{itemize}
The f\mbox{}irst property is true because if $P$ would have two dif\mbox{}ferent minimal elements, say \a  ~ and \b, then  (ASL 2) forces $\a\b = 0$. For a proof of the second property see \cite[(5.10)]{bv}. The third remark is the immediate consequence of the fact that the standard monomials (which correspond to the multichains of $P$) generate $A$ as a $K$-vector space.

The Veronese algebra of an integral algebra is again a domain, we know that $\dim A = \dim\A{d}$ and by def\mbox{}inition  $(\A{d})_i = A_{di}$, so a  candidate for $\P{d}$ should have the following properties:
\begin{itemize}
\item[(1c)] If $P$ has a unique minimal element, so should \P{d}.
\item[(2c)] \rank($P$) = \rank(\P{d}).
\item[(3c)]  $|\{md$-multichains in $P\}|$  = $|\{m$-multichains in \P{d} $\}|$ for all $m \ge 1$.
\end{itemize}
A poset construction with the above properties that works for every poset is not known to us.
A construction that has properties (2c) and (3c) is the zig-zag poset $Z_d(P)$, which is obtained  in the following way. Let $P = \{ \a_1, \ldots, \a_n\}$ be a poset. Given $d \ge 2$ a positive integer, one can def\mbox{}ine:
\[
Z_d(P)=\{ (\a_{i_1},\ldots,\a_{i_d}) ~:~ \a{_j} \in P, ~\forall~ j~ \textup{and} ~\a_{i_1}\le\ldots\le\a_{i_d}\}
\] 
and say that:
\[
\begin{array}{cccl}
 (\a_{i_1},\ldots,\a_{i_d}) \le (\b_{i_1},\ldots,\b_{i_d}) & \iff &   & \a_{i_1} \le \b_{i_1}, \\
                                                                                           &    & \textrm{and} & \a_{i_2} \ge \b_{i_2}, \\
                                                                                           &    & \textrm{and} & \a_{i_3} \le \b_{i_3}, \\                                                                                           
                                                                                           &    & \textrm{and} & \a_{i_4} \ge \b_{i_4}, \\       
                                                                                           &    &  & \ldots \textup{etc.}
\end{array}
\]            
The correspondence between the $md$-multichains of $P$ and the  $m$-multichains in \P{d} can be seen easily in the following picture. Suppose $m = 3$ and $d=4$:
\[
\begin{array}{ccccc}
\a_1& \le & \b_1& \ \le & \c_1\\      
\wedge\textup{\textsf{\small I}}&&\wedge\textup{\textsf{\small I}}&&\wedge\textup{\textsf{\small I}}\\
\a_2& \ge & \b_2& \ \ge & \c_2\\      
\wedge\textup{\textsf{\small I}}&&\wedge\textup{\textsf{\small I}}&&\wedge\textup{\textsf{\small I}}\\
\a_3& \le & \b_3& \ \le & \c_3\\      
\wedge\textup{\textsf{\small I}}&&\wedge\textup{\textsf{\small I}}&&\wedge\textup{\textsf{\small I}}\\
\a_4& \ge & \b_4& \ \ge & \c_4\\      
\end{array}
\]                                            
The $md$-multichain  of $P$ that can be associated to the $d$-multichain of $Z_d(P)$,  $\a \le \b \le \c$ is: $\a_1 \le \b_1 \le \c_1 \le \c_2 \le \b_2 \le \ldots \le \c_3 \le \c_4 \le \b_4 \le \a_4$.  The other way around should also be clear now.
So $Z_d(P)$ satisf\mbox{}ies (3c).  It is  easy to see  that also (2c) is satisf\mbox{}ied. Unfortunately (1c) is almost never satisf\mbox{}ied, in the sense that if $d\ge3$, then $Z_d(P)$ has at least two minimal elements. The only case in which $Z_d(P)$ satisf\mbox{}ies also (1c) is when $d=2$ and $P$ has also a unique maximal element. However, we will show in the remaining part of this section that in two particular cases $Z_2(P)$ is the right choice for the supporting poset of the second Veronese. These cases are the discrete ASLs over any poset and the Hibi rings over a distributive lattice.

Let us f\mbox{}irst f\mbox{}ix some more  terminology. Let $P$ be a poset and $\a, \b \in P$. Whenever the right hand side exists, we use the following notation
\begin{eqnarray*}
\a \wedge \b &=& \sup\{m \in P ~:~ m \le \a~ \textup{and}~ m \le \b \},\\
\a \vee \b &=& \inf\{M \in P ~:~ M \ge \a~ \textup{and}~ M \ge \b \}.
\end{eqnarray*}
When these elements exist, they are called  \emph{greatest lower bound} or \emph{inf\mbox{}imum}, respectively \emph{least upper bound} or \emph{supremum}.  
A poset $P$ in which for any two  $\a, \b \in P$, the elements $\a \wedge \b$ and $\a \vee \b$ exist is called a \emph{lattice}. 
A lattice $P$ is called \emph{distributive} if the operations def\mbox{}ined by $\wedge$ and $\vee$ are distributive to each other. In other words,  if for any $\a, \b,\c \in P$ we have
\begin{eqnarray*}
\a \wedge (\b \vee \c) &=& (\a \wedge \b) \vee (\a \wedge \c)  \textup{~~and} \\
\a \vee (\b \wedge \c)& = &(\a \vee \b) \wedge (\a \vee \c).  
\end{eqnarray*} 

As we already said in the first section, on every poset $P$ we can construct the discrete ASL $K[P]/I_P$. The straightening relations of this algebra are 
\[ x_ix_j=0,~~~\forall~x_i,x_j\in P \textup{~with~}x_i\not\sim x_j.\]
This algebra plays a special role as, for any other ASL on $P$ presented as  $K[P]/I$ and for any reversed lexicographic term ordering $\tau$ corresponding to a linear extension of the partial order on $P$, we have
\[\ini_\tau(I)= I_P=(x_ix_j~:~x_i\not\sim x_j).\]
We will prove the following theorem regarding the discrete ASL of any poset $P$.
\begin{theorem}\label{asldiscr}
Let $P$ be a poset and $A$ the discrete ASL on $P$ over a f\mbox{}ield $K$. The second Veronese algebra $\A{2}$ is a homogeneous ASL on $Z_2(P)$ over $K$.
\end{theorem}

\begin{proof}
Let us denote $P = \{x_1,\ldots,x_n\}$, so the straightening relations of $A$ are $x_ix_j = 0$ if $x_i \not\sim x_j$. The vertices of $Z_2(P)$ are the standard monomials of $P$ of degree two, which clearly generate  $\A{2}$ as a $K$-algebra. As the standard monomials in $Z_2(P)$ can be also seen as standard monomials in $A$, it is again clear that they form a $K$-vector space basis of $\A{2}$. For any two incomparable elements $x_ix_j\not\sim x_kx_l$ of $Z_2(P)$ we def\mbox{}ine the straightening laws in the following way:
\begin{equation*}
(x_ix_j)(x_kx_l)=\left\{\begin{array}{ll}
                               0 & \textup{ if }\{x_i,x_j,x_k,x_l\}\textup{ is not totally ordered,} \\
                               (x_{i_1}x_{l_1})(x_{j_1}x_{k_1}) &\textup{ if  }\{x_i,x_j,x_k,x_l\}\textup{ is  totally ordered,}
                               \end{array}\right.
\end{equation*}
where $x_{i_1}\le x_{l_1}\le x_{j_1}\le x_{k_1}$ and $\{x_{i_1}, x_{l_1}, x_{j_1}, x_{k_1}\} = \{x_i,x_j,x_k,x_l\}$  as multisets. We just need to check now that the ASL we def\mbox{}ined is actually the second Veronese subring of $A$.

 It is easy to see that the relations among the canonical algebra generators of \A{2} are given by the $2\times 2$ minors of the symmetric matrix
\begin{equation*}
X = \left(\begin{array}{cccc}
                 x_1^2 &  x_1x_2  & \ldots & x_1x_n\\
                 x_1x_2& x_2^2 & \ldots & x_2x_n\\
                 \vdots& \vdots & &\vdots\\
                 x_1x_n&x_2x_n& \ldots &x_n^2\\
                 \end{array}\right),
\end{equation*}                 
where the monomials that are not standard are replaced by 0. For a f\mbox{}ixed set of four dif\mbox{}ferent variables (elements of $P$) there are  three dif\mbox{}ferent minors involving precisely those variables. It is easy to check that two of them correspond to straightening relations as above and the third one is superf\mbox{}luous, in the sense that it can be obtained as a linear combination of the other two. With this in mind and noticing that if two of the variables coincide, there exists only one relation, it is straight forward to check that the theorem holds.
\end{proof}
We present now an example of the ASL structure given in Theorem  \ref{asldiscr}. 
\begin{example}
In the following picture we can see on the left the non-pure poset $P$ and on the right hand side $Z_2(P)$: \\

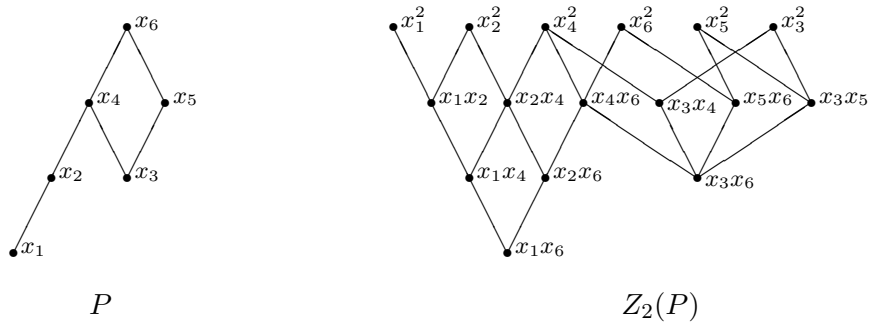
 \setlength{\unitlength}{1mm}
\begin{picture}(50,41)

\put(10,10){\line(1,2){15}}
\put(25,20){\line(1,2){5}}
\put(25,20){\line(-1,2){5}}
\put(30,30){\line(-1,2){5}}

\put(75,10){\line(1,2){15}}
\put(75,10){\line(-1,2){15}}
\put(70,20){\line(1,2){10}}
\put(80,20){\line(-1,2){10}}
\put(65,30){\line(1,2){5}}
\put(85,30){\line(-1,2){5}}

\put(100,20){\line(-3,2){15}}
\put(100,20){\line(-1,2){5}}
\put(100,20){\line(1,2){5}}
\put(100,20){\line(3,2){15}}

\put(95,30){\line(-3,2){15}}
\put(95,30){\line(3,2){15}}
\put(105,30){\line(-3,2){15}}
\put(105,30){\line(-1,2){5}}
\put(115,30){\line(-3,2){15}}
\put(115,30){\line(-1,2){5}}

\put(20,2){$P$}
\put(90,2){$Z_2(P)$}

{\footnotesize
\put(10,10){\circle*{1}}  \put(11,10){$x_1$}
\put(75,10){\circle*{1}}  \put(76,10){$x_1x_6$}
\put(15,20){\circle*{1}}  \put(16,20){$x_2$}
\put(25,20){\circle*{1}}  \put(26,20){$x_3$}
\put(70,20){\circle*{1}}  \put(71,20){$x_1x_4$}
\put(80,20){\circle*{1}}  \put(81,20){$x_2x_6$}
\put(100,20){\circle*{1}}  \put(101,19){$x_3x_6$}
\put(20,30){\circle*{1}}  \put(21,30){$x_4$}
\put(30,30){\circle*{1}}  \put(31,30){$x_5$}
\put(65,30){\circle*{1}}  \put(66,30){$x_1x_2$}
\put(75,30){\circle*{1}}  \put(76,30){$x_2x_4$}
\put(85,30){\circle*{1}}  \put(86,30){$x_4x_6$}
\put(95,30){\circle*{1}}  \put(96,29){$x_3x_4$}
\put(105,30){\circle*{1}}  \put(106,30){$x_5x_6$}
\put(115,30){\circle*{1}}  \put(116,30){$x_3x_5$}
\put(25,40){\circle*{1}}  \put(26,40){$x_6$}
\put(60,40){\circle*{1}}  \put(61,40){$x_1^2$}
\put(70,40){\circle*{1}}  \put(71,40){$x_2^2$}
\put(80,40){\circle*{1}}  \put(81,40){$x_4^2$}
\put(90,40){\circle*{1}}  \put(91,40){$x_6^2$}
\put(100,40){\circle*{1}}  \put(101,40){$x_5^2$}
\put(110,40){\circle*{1}}  \put(111,40){$x_3^2$}
}
\end{picture}\\
It is easy to see that each of the straightening relations def\mbox{}ined in the proof of Theorem \ref{asldiscr} can be found as a $2\times 2$ minor of the matrix $X$ bellow. For instance $(x_1x_4)(x_2x_6) = (x_1x_6)(x_2x_4)$ corresponds to the minor $[ 1,2 | 4,6]$ (that is the minor obtained by taking rows 1 and 2 and the columns 4 and 6), or $[4,6|1,2]$. Notice that the straightening laws that have zero on the right hand side, are actually forced by partial order on $Z_2(P)$. For example $(x_1x_2)(x_3x_4) =0 $ by def\mbox{}inition because $x_1\not\sim x_3$, but there also is no other choice as no element of $Z_2(P)$ is less or equal to  both $(x_1x_2)$ and $(x_3x_4)  $ simultaneously. 
\begin{equation*}
X = \left(\begin{array}{cccccc}
                 x_1^2 &  x_1x_2  &  0 & x_1x_4& 0&x_1x_6\\
                 x_1x_2 &  x_2^2  & 0 & x_2x_4&0&x_2x_6\\
                 0 & 0 & x_3^2 & x_3x_4&x_3x_5&x_3x_6\\
                 x_1x_4 &  x_2x_4  & x_3x_4 & x_4^2&0&x_4x_6\\
                 0 &  0  & x_3x_5 & 0&x_5^2&x_4x_6\\
                 x_1x_6 &  x_2x_6  & x_3x_6 & x_4x_6&x_5x_6&x_6^2\\
\end{array}\right).
\end{equation*} 
If we consider the set of variables $\{x_1,x_2,x_4,x_6\}$, the dif\mbox{}ferent minors that involve them are $[1,2|4,6]$, $[1,4|2,6]$ and $[1,6|2,4]$. The f\mbox{}irst one corresponds to the straightening relation for $(x_1x_4)(x_2x_6)$, the second one to the straightening of $(x_1x_2)(x_4x_6)$. The third one is not a straightening relation but we  have
\[ [1,6|2,4]=[1,4|2,6] -[1,2|4,6].\]

\end{example}

Another interesting and at the same time dif\mbox{}f\mbox{}icult  problem regarding ASLs is to give a description of  the \emph{integral} posets. We say that a poset $P$ is integral if there exists an ASL on $P$ that is an integral algebra. We have seen that a necessary condition for $P$ is to have a unique minimal element.
In \cite{hibi},  Hibi shows that every distributive lattice is integral. He constructs for any distributive lattice $P$ an ASL that is integral as an algebra, which is now called the Hibi ring on $P$. The generators of this $K$-algebra are the vertices of the lattice $P$ and the straightening laws are the so called Hibi relations:
\[ x_\a x_\b = x_{\a\wedge\b}x_{ \a\vee\b}, \quad \forall~x_a\not\sim x_\b \in P.\]
From this point on, in order to simplify notation, we will use only  greek letters $\a,\b,\c\ldots$ for the elements of $P$ and for the variables of the polynomial ring. We will prove  the following.

\begin{theorem}\label{hibirings}
Let $P$ be a distributive lattice and $A$ be the ASL on $P$ given by the Hibi relations. Then \A{2} is  an ASL over 
$Z_2(P)$ with the following structure: the vertices of $Z_2(P)$ are the standard monomials of degree 2 in $A$ and the straightening laws are:
\begin{equation}\label{strlaw}
(\a \b)(\c \d) = [(\a\wedge\c)(\b\vee\d)][((\a\wedge\d)\vee(\b\wedge\c))((\a\vee\d)\wedge(\b\vee\c))],
\end{equation}
$\quad \forall~ \a, \b, \c, \d \in P, ~\textup{with}~ \a \le \b,~ \c \le\d ~\textup{and}~ \a\b \not\sim\c\d$.
\end{theorem}
In many cases the right hand  side in \eqref{strlaw} can be presented in a shorter form, but this presentation has the advantage to include all cases. For  example, if the set $\{\a,\b,\c,\d\}$ is totally ordered, but $(\a\b) \not\sim (\c\d)$, then it is easy to check that (\ref{strlaw}) gives us:
\[(\a\b)(\c\d) = (\a_0\d_0)(\b_0\c_0),\]
where $\a_0 \le \b_0 \le c_0 \le \d_0$ and $\{\a,\b,\c,\d\} = \{\a_0,\b_0,\c_0,\d_0\}$ as multisets. Also if $\a\vee\c$ and $\b\wedge\d$ are comparable, then (\ref{strlaw}) is actually:
\[(\a \b)(\c \d) = [(\a\wedge\c)(\b\vee\d)][(\a\vee\c)(\b\wedge\d)].\]
\begin{proof}
We have to check f\mbox{}irst that the structure described above is  an ASL structure and second that this ASL is  \A{2}. The condition (ASL 0) is satisf\mbox{}ied by definition. 
The fact that the standard monomials in $Z_2(P)$ are  the standard monomials in $P$ of even degree follows from the correspondence between $m$-multichains in $Z_2(P)$ and $2m$-multichains in $P$, so (ASL 1) is satisf\mbox{}ied.
To prove that (ASL 2) holds we have to check that:
\begin{itemize}
\item[1.] $(\a\wedge\c)(\b\vee\d)$ and $\big((\a\wedge\d)\vee(\b\wedge\c)\big)\big((\a\vee\d)\wedge(\b\vee\c)\big)$ are actually vertices in $Z_2(P)$, (that is multichains of length 2 in $P$), 
\item[2.] that the right hand side is a standard monomial in $Z_2(P)$, that is 
\[(\a\wedge\c)(\b\vee\d) \le \big((\a\wedge\d)\vee(\b\wedge\c)\big)\big((\a\vee\d)\wedge(\b\vee\c)\big),\]
\item[3.]  $(\a\wedge\c)(\b\vee\d) \le (\a\b)$ and  $(\a\wedge\c)(\b\vee\d) \le (\c \d).$
\end{itemize}
Here is a picture of the elements of $P$ that we are interested in and the order  relations between them that always hold:\\
\setlength{\unitlength}{0.5mm}
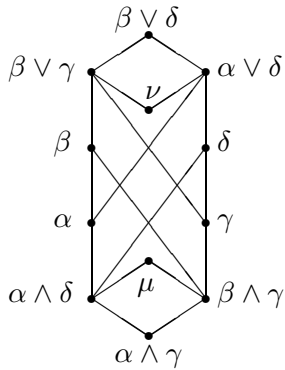
\begin{picture}(50,100)(-32,0)

\put(10,20){\line(0,1){60}}
\put(10,20){\line(3,4){30}}
\put(10,20){\line(3,2){15}}
\put(10,20){\line(3,-2){15}}
\put(10,40){\line(3,4){30}}
\put(10,60){\line(3,-4){30}}
\put(10,80){\line(3,-4){30}}
\put(10,80){\line(3,-2){15}}
\put(10,80){\line(3,2){15}}
\put(40,20){\line(0,1){60}}
\put(40,20){\line(-3,2){15}}
\put(40,20){\line(-3,-2){15}}
\put(40,80){\line(-3,-2){15}}
\put(40,80){\line(-3,2){15}}

\put(10,20){\circle*{2}}  \put(-12,19){$\a \wedge \d$}
\put(10,40){\circle*{2}}  \put(0,39){$\a$}
\put(10,60){\circle*{2}}  \put(0,59){$\b$}
\put(10,80){\circle*{2}}  \put(-12,79){$\b \vee \c$}
\put(25,10){\circle*{2}}  \put(16,3){$\a \wedge \c$}
\put(25,30){\circle*{2}}  \put(22,22){$\mu$}
\put(25,70){\circle*{2}}  \put(24,73){$\nu$}
\put(25,90){\circle*{2}}  \put(16,92){$\b \vee \d$}
\put(40,20){\circle*{2}}  \put(43,19){$\b \wedge \c $}
\put(40,40){\circle*{2}}  \put(43,39){$\c$}
\put(40,60){\circle*{2}}  \put(43,59){$\d$}
\put(40,80){\circle*{2}}  \put(43,79){$\a \vee \d$}

\put (110,22){$\mu = (\a \wedge\d) \vee (\b \wedge \c)$}
\put (110,73){$\nu = (\b \vee\c) \wedge (\a \vee \d)$}
\end{picture}\\
To check the f\mbox{}irst point, we will show how this straightening law came up. Suppose that, like in the above picture, $\a \not\sim \c$ and $\b \not\sim \d$. Notice that this is not a restriction, as in general $\a\c = (\a \wedge \c)(\a \vee \c)$ also when $\a$ and $\c$ are comparable.  We use the Hibi relations in $A$ to "straighten" $\a\c$ and $\b\d$. It is easy to see  that $\a\wedge\c \le \b \vee \d$. The problem is that  $\a\vee\c$ and $\b\wedge\d$ are not always comparable, which means $(\a\vee\c)(\b\wedge\d)$ is not always an element of $Z_2(P)$.  Suppose they are not comparable. We "straighten" also this product using the Hibi relations. So we get the following:
\[
(\a\vee\c)(\b\wedge\d)  =   \big((\a\vee\c) \wedge (\b\wedge\d)\big)\big((\a\vee\c)\vee(\b\wedge\d)\big).
\]
Now we just have to show that the f\mbox{}irst element on the  right hand side is $\mu$ and the second one $\nu$. Just by using  distributivity and the fact that $\a \le \b$ and $\c \le \d$ we get:
\[
\begin{array}{ccl}
(\a\vee\c) \wedge (\b\wedge\d)& = & \big((\b\wedge\d)\wedge\a\big) \vee \big((\b\wedge\d)\wedge\c\big)\\
                                                    & = & (\d\wedge\a) \vee (\b\wedge\c)\\
                                                    & = & \mu\\
                                                 &  &               \\
(\a\vee\c) \vee (\b\wedge\d)& = & \big((\a\vee\c)\vee\b\big) \wedge \big((\a\vee\c)\vee\d\big)\\
                                                    & = & (\c\vee\b) \wedge (\a\vee\d)\\
                                                    & = & \nu
                                             
\end{array}
\]
So $\mu\nu$ is also a standard monomial and the  law that we gave is actually a relation in $A$.
    
To prove 2. we just have to look at the Hasse diagram above and notice that as 
\[
\a \wedge \c \le \a \wedge \d ~\textup{and}~ \a \wedge \c \le \b \wedge \c 
\]                                
we get that $\a \wedge \c \le \mu$. Using the same way of reasoning we also get that $\b \vee \d \ge \nu$, so 2. holds.
It is clear that the third point also holds.

The straightening laws that we have def\mbox{}ined in \eqref{strlaw} can be divided into two types: 
\begin{itemize}
\item[Type 1.]Straightening relations in $A$,  when $\{\a,\b,\c,\d\}$ is not totally ordered. 
\item[Type 2.]Veronese type relations, which are 0 when seen as elements of $A$, when $\{\a,\b,\c,\d\}$ is totally ordered.
 \end{itemize}
 As exactly these are also the relations that def\mbox{}ine \A{2}, we can conclude that the ASL we have constructed is actually \A{2}.
\end{proof}
\noindent As  we already said, the Hibi rings were introduced as examples of ASL domains, thus proving that all distributive lattices are integral posets. As an immediate  consequence of Theorem \ref{hibirings} we have:
\begin{corollary}
The second zig-zag poset of a distributive lattice is an integral poset.
\end{corollary}

\section{A Poset  Construction in Dimension Three}

Let $P$ be a   poset of rank at most three. Denote the   minimal elements  of $P$ by $\m_1, \m_2, \ldots, \m_r$ and  let $d \ge 2$ be a positive integer. We will construct a poset \P{d} that has the combinatorial properties (1c), (2c) and (3c) described in the previous section.

Let the elements of  \P{d} be the $d$-multichains in $P$. Let $\a =  (\a_1,\ldots.\a_d)$ and $\b = (\b_1,\ldots.\b_d)$ be two such multichains. Recall that this means  $\a_1 \le \ldots \le \a_d$ and $\b_1 \le \ldots \le \b_d.$ 
For each multichain \a~ we def\mbox{}ine:\[ \vv(\a) = (\hght(\a_1), \hght(\a_2)-\hght(\a_1),\ldots, \hght(\a_d)-\hght(\a_{d-1})).\]
We say that $\a \le \b$ if the following two conditions hold:
\begin{itemize}
\item[1.]
The set $\{ a_1,\ldots.\a_d,\b_1,\ldots.\b_d\}$ is totally ordered. 
\item[2.] 
$\vv(\a) \le \vv(\b)$ component-wise.
\end{itemize}
First of all notice that the two conditions above imply that $\a_i \le \b_i$ for every $i=1,\ldots,d$. The converse does not hold, meaning that the relation above is not the component-wise order on the set of $d$-multichains in $P$. For instance, if $P =\{0,1,2\}$ with the natural order, the 2-multichain $(2,2)$ is component-wise larger than the 2-multichain $(0,1)$, but  $\vv((2,2)) = (2,0)$ and $ \vv((0,1)) = (0,1)$ are not comparable, so condition 2. does not hold.

In general, for a vector $v = (v_1,\ldots,v_n)$ denote by $|v|= \sum_{i=1}^{i=n}v_i$. In our case, the fact that $P$ has rank $3$ implies that for every $d$-multichain \a~, $|\vv(\a)| \le 2$. It is easy to see that if $\a < \b$, then $|\vv(\a)| < |\vv(\b)|$. 
Also notice that the only $d$-multichains $\a_{i}$ with $|\vv(\a_{i})| = 0$ are $\a_{i} = (\m_i,\ldots,\m_i)$, for some minimal element of $P$.  This fact will guarantee that if $P$ has a unique minimal element, then $\P{d}$ has a unique minimal element as well. But f\mbox{}irst we need to check that we have def\mbox{}ined a partial order. 
\begin{lemma}\label{po}
If $P$ is a poset with $\rank(P)\le3$,  the above relation   is a partial order on $\P{d}$. 
 \end{lemma}
\begin{proof}
 Ref\mbox{}lexivity is obvious. As two elements of the same height in $P$ are either not comparable or equal, antisymmetry follows as well. 
 To check transitivity it is enough to suppose that all inequalities are strict. Let $\a,\b,\c$ be $d$-multichains such that $\a < \b$ and $\b < \c$. Then we also have $|\vv(\a)| < |\vv(\b)|$ and $|\vv(\b)| < |\vv(\c)|$. As $|\vv(\a)|,|\vv(\b)|, |\vv(\c)| \in \{0,1,2\}$ this implies $|\vv(\a)| = 0$, so  we get $\a = \a_{i} =(\m_i,\ldots,\m_i)$ for some minimal element $\m_i$. By the f\mbox{}irst condition we obtain that $\m_i$ is also the minimal element of the totally ordered set $\{\m_i,\b_1,\ldots,\b_d\}$, where $\b=(\b_1,\ldots,\b_d)$. As  we also have $\b\le\c$ component-wise, we obtain that the set $\{\m_i,\c_1,\ldots,\c_d\}$ is totally ordered. Clearly $\vv(\a) \le \vv(\c)$ by the transitivity of the component-wise ordering, so we obtain the transitivity of the relation we def\mbox{}ined.
 \end{proof}
\noindent The above proof  obviously depends on the fact that the rank of $P$ is three. But this is actually a necessary condition for Lemma \ref{po}, in the sense that for $\rank(P) > 3$ transitivity may fail.\\

From now on we will consider on the set $\P{d}$ only the partial order of Lemma \ref{po}. In general, for a positive integer $m$ and a poset $P$, we denote: 
 \[M_m(P)=\{ m\textup{-multichains in }P\}.\] 
 For an $m$-multichain \a~in $P$ denote by $\supp(\a)$ the set of vertices that appear in \a. If $\a'$~is a multichain in \P{d} then by $\supp_P(\a')$ we denote the set of vertices of $P$ that appear in any of  the $d$-multichains that $\a'$~is made of.  For example, if $P =\{0,1,2\}$ with the natural order, then
 \begin{eqnarray*}
 M_2(P)& = &\{(0,0),(0,1),(1,1),(0,2),(1,2),(2,2)\},\\
 \supp((1,1,2,2))& = & \{1,2\},\\
 \supp_P(((0,0),(0,1),(0,2)))&=&\{0,1,2\}.
 \end{eqnarray*}
 Before we prove that \P{d} has the desired combinatorial properties, we will prove the following remark.

\begin{remark}\label{bij}
1.  If $P_0 = \{0,1,2\}$ with the natural order, then $P_0^{(d)} \cong H_3(d)$.\\
2. There exists a bijection, say
 $ f_{P_{0},d,m} :  M_{md}(P_0) \longrightarrow  M_m(P_0^{(d)}), $
 such that for any $\a \in M_{md}(P_0)$ we have $\supp(\a) = \supp_{P_0}(f_{P_0,d,m}(\a))$.
\end{remark}
 \begin{proof}
In the f\mbox{}irst part,  the isomorphism of posets  is given by:
 \[ \a = (\a_1,\ldots,\a_d) \longmapsto \vv(\a)=(\a_1,\a_2-\a_1,\ldots,\a_d-\a_{d-1}), \] 
its inverse being:
 \[ H_3(d) \ni v= (v_1,\ldots,v_d) \longmapsto (v_1, v_1 +v_2,\ldots, \sum_{i=1}^{i=d}v_i).\]
For the second part, we already know by the ASL structure on $H_3(d)$ of the polynomial ring in three variables that
 $|M_{md}(P_0)|=|M_m(P_0^{(d)})|$ for every $m$.  It is easy to check that  for every subposet of $Q \subset P_0$ we have that $Q^{(d)}$ is a subposet of $P_0^{(d)}$ with the canonical embedding.   So we can construct $f_{P_0,d,m}$ step by step, starting  with $|\supp(\a)| =1$, which correspond to subposets of rank one.
\end{proof}

We will show now an example of how the bijection $f_{P_0,d,m}$ above can be constructed. We will also see that if $P_0$ is the chain of length 4, the fact that for every subposet $Q \subset P_0$ also $Q^{(d)}$ is a subposet of $P_0^{(d)}$ canonically, no longer holds. 

\begin{example} As in Remark \ref{bij}, let $P_0=\{0,1,2\}$ with the natural order. We will construct
\[f_{P_0,2,2} : M_{4}(P_0) \longrightarrow  M_2(P_0^{(2)}).\]
We start with the 4-multichains  supported on one element, that is: $(0,0,0,0), (1,1,1,1)$ and $(2,2,2,2)$. Notice that $\{i\}$ is a subposet of $P_0$ and $\{i\}^{(2)} = \{(i,i)\} \subset P_0^{(2)}$. So in this case there is no other choice than:
\[f_{P_0,2,2,}((i,i,i,i))= ((i,i),(i,i)),\quad \forall~ i=0,1,2.\] 
We will now consider the 4-multichains supported on two elements, and divide them in groups corresponding to the rank two subposet of $P_0$ that contains them. In particular, the 4-multichains  contained in $Q=\{0,1\}$ and with support $\{0,1\}$ are $(0,0,0,1),(0,0,1,1)$ and $(0,1,1,1)$. We have $Q^{(2)} = \{(0,0),(0,1),(1,1)\}$ which is a subposet of $P_0^{(2)}$. The 2-multichains of $Q^{(2)}$ supported on $\{0,1\}$ are $((0,0),(0,1)),((0,0),(1,1))$ and $((0,1),(0,1))$. As one can see, there is no canonical  way in which to def\mbox{}ine the function, but any bijection between the two sets satisf\mbox{}ies the required conditions. We proceed in the same way with the 4-chains supported on $\{0,2\}$ and $\{1,2\}$. As we know a priori that $|M_{4}(P_0)|=|M_2(P_0^{(2)})|$ we  obtain also
\[  |\{\a \in M_4(P_0)~:~\supp(\a)=\{0,1,2\}\}| = |\{\a'\in M_2(P_0^{(2)})~:~\supp_{P_0}(\a')=\{0,1,2\}\}|,\]
so again any bijection between the two sets works.

Notice that if $P=\{0,1,2,3\}$ with the natural order, the subposet $Q=\{0,2,3\}$ has the property that the subposet of $P^{(2)}$ induced by the 2-multichains of $Q$ is not isomorphic to $Q^{(2)}$. For example $(0,2)\not\sim(2,3)$ in $P^{(2)}$, while $(0,2)<(2,3)$ in $Q^{(2)}$.
\end{example}

\begin{proposition}
Let $P$ be a  poset of rank at most three. Then for any $d\ge1$, the poset $\P{d}$ constructed above satisf\mbox{}ies:
\begin{itemize}
\item[\textup{(1c)}] If $P$ has a unique minimal element, \P{d} has a unique minimal element.
\item[\textup{(2c)}] \rank($P$) = \rank(\P{d}).
\item[\textup{(3c)}]  $|M_{md}(P)|$  = $|M_m(\P{d})|$ for all $m \ge 1$.
\end{itemize}
\end{proposition}
Before we come to the actual proof we have one f\mbox{}inal observation. Any poset $P$ can be seen as the union of its maximal chains. This union is not disjoint, but the construction of \P{d}~ can be done on each such maximal chain $C$ and then \P{d}~ will be the union of the $C^{(d)}$-s. In the following f\mbox{}igure we present an example of how the construction of $\P{d}$ can be done chain-wise. \\
\setlength{\unitlength}{1mm}
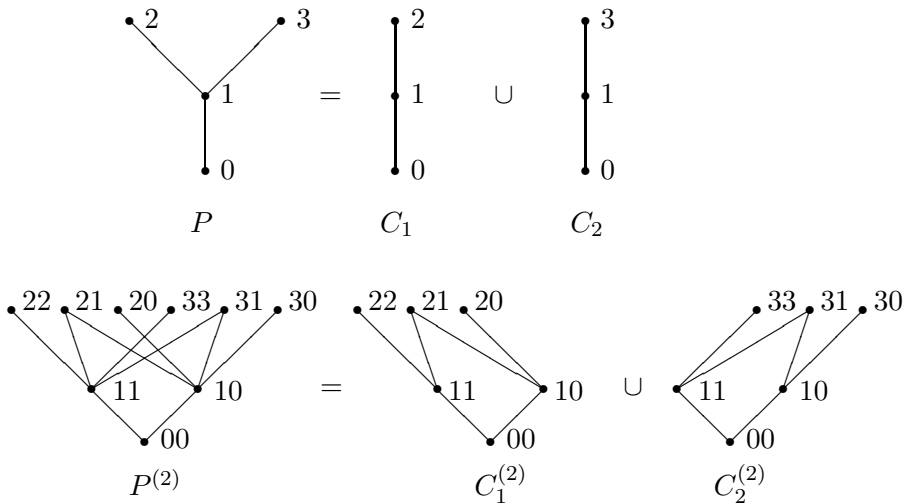
\begin{picture}(120,40)(-30,0)
\put(20,10){\circle*{1}}   \put(22,9){$0$}
\put(20,20){\circle*{1}}   \put(22,19){$1$} 
\put(10,30){\circle*{1}}   \put(12,29){$2$}
\put(30,30){\circle*{1}}   \put(32,29){$3$}
\put(18,2){$P$}
\put(35,19){=}

\put(20,10){\line(0,1){10}}
\put(20,20){\line(-1,1){10}}
\put(20,20){\line(1,1){10}}
\put(45,10){\line(0,1){20}}
\put(70,10){\line(0,1){20}}

\put(45,10){\circle*{1}}  \put(47,9){$0$}
\put(45,20){\circle*{1}}  \put(47,19){$1$}
\put(45,30){\circle*{1}}  \put(47,29){$2$}
\put(43,2){$C_1$}

\put(58,19){$\cup$}

\put(70,10){\circle*{1}}  \put(72,9){$0$}
\put(70,20){\circle*{1}}  \put(72,19){$1$}
\put(70,30){\circle*{1}}  \put(72,29){$3$}
\put(68,2){$C_2$}
\end{picture}\\

\setlength{\unitlength}{0.7mm}
\begin{picture}(120,40)(-24,0)

\put(30,10){\circle*{1.43}}   \put(33,9){$00$}
\put(20,20){\circle*{1.43}}   \put(24,18){$11$} 
\put(40,20){\circle*{1.43}}   \put(43,18){$10$}
\put(5,35){\circle*{1.43}}   \put(7,35){$22$}
\put(15,35){\circle*{1.43}}   \put(17,35){$21$}
\put(25,35){\circle*{1.43}}   \put(27,35){$20$}
\put(35,35){\circle*{1.43}}   \put(37,35){$33$}
\put(45,35){\circle*{1.43}}   \put(47,35){$31$}
\put(55,35){\circle*{1.43}}   \put(57,35){$30$}

\put(27,0){$P^{(2)}$}
\put(63,19){=}

\put(30,10){\line(-1,1){25}}
\put(30,10){\line(1,1){25}}
\put(20,20){\line(-1,3){5}}
\put(20,20){\line(1,1){15}}
\put(20,20){\line(5,3){25}}
\put(40,20){\line(1,3){5}}
\put(40,20){\line(-1,1){15}}
\put(40,20){\line(-5,3){25}}

\put(95,10){\circle*{1.43}}   \put(98,9){$00$}
\put(85,20){\circle*{1.43}}   \put(87,18){$11$} 
\put(105,20){\circle*{1.43}}   \put(107,18){$10$}
\put(70,35){\circle*{1.43}}   \put(72,35){$22$}
\put(80,35){\circle*{1.43}}   \put(82,35){$21$}
\put(90,35){\circle*{1.43}}   \put(92,35){$20$}

\put(95,10){\line(-1,1){25}}
\put(95,10){\line(1,1){10}}
\put(85,20){\line(-1,3){5}}
\put(105,20){\line(-1,1){15}}
\put(105,20){\line(-5,3){25}}

\put(92,0){$C_1^{(2)}$}

\put(120,19){$\cup$}

\put(140,10){\circle*{1.43}}   \put(143,9){$00$}
\put(130,20){\circle*{1.43}}   \put(134,18){$11$} 
\put(150,20){\circle*{1.43}}   \put(153,18){$10$}
\put(145,35){\circle*{1.43}}   \put(147,35){$33$}
\put(155,35){\circle*{1.43}}   \put(157,35){$31$}
\put(165,35){\circle*{1.43}}   \put(167,35){$30$}

\put(137,0){$C_2^{(2)}$}

\put(140,10){\line(-1,1){10}}
\put(140,10){\line(1,1){25}}
\put(130,20){\line(1,1){15}}
\put(130,20){\line(5,3){25}}
\put(150,20){\line(1,3){5}}

\end{picture}\\

\begin{proof}
As we have already noticed, the minimal elements of $\P{d}$ are of the form $(\m_i,\ldots,\m_i)$ for all minimal elements $\m_i\in P$. This implies (1c). It is also clear that \rank($P$) = \rank(\P{d}). So we just need to def\mbox{}ine a bijection from $M_{md}(P)$  to $M_m(\P{d})$. To this aim we will use the observation that $\P{d}$ can be constructed chain-wise.

 We  f\mbox{}ix for each maximal chain $C$ in $P$ a bijection $f_{C,d,m}$ as in Remark \ref{bij}. It is easy to see this can be done in a coherent way, in the sense that if $\a \in C\cap C'$, then $f_{C,d,m}(\a) = f_{C',d,m}(\a).$
Let $\a \in M_{md}(P)$. We def\mbox{}ine $F :  M_{md}(P) \longrightarrow M_m(\P{d})$ as follows 
\[ F(\a) = f_{C,d,m}(\a) \in C^{(d)} \subset \P{d},  \]
 where  $C$ is a maximal chain such that $\supp(\a) \subseteq C$.
From the way we chose $f_{C,d,m}$,  we can deduce that $F(\a)$ is well def\mbox{}ined.
The function $F$ is bijective because it has an inverse $F^{-1} :  M_{m}(\P{d}) \longrightarrow M_{md}(P)$ given by
 \[ F^{-1}(\b) = f_{C,d,m}^{-1}(\b) \in C \subset P, \]
 where $\b\in  M_{m}(\P{d})$ and $C \subset P$ is a maximal chain such that $\b \in C^{(d)}$. The same arguments as above tell us that also $F^{-1}$ is well def\mbox{}ined. 
\end{proof}

\end{document}